\documentclass{amsart}
\title{Good reduction and canonical heights of subvarieties}

\author{Benjamin Hutz}
\address{
Department of Mathematics and Computer Science\\
Saint Louis University\\
St. Louis, MO}
\email{hutzba@slu.edu}

\thanks{The author is partially supported by NSF Grant DMS-1415294.}



\usepackage{color}
\usepackage{amssymb,amsmath,amsthm,fullpage}
\usepackage[all]{xy}  
\usepackage{url}
\usepackage{rotating} 
\usepackage{tikz}
\usepackage{textcomp,listings}
\usepackage{versions}

\definecolor{green}{rgb}{0,0.5,0}
\definecolor{dkgreen}{rgb}{0,0.6,0}
\definecolor{gray}{rgb}{0.5,0.5,0.5}
\definecolor{mauve}{rgb}{0.58,0,0.82}
\lstnewenvironment{python}[1][]{
\lstset{
  frame=single,                   
  language=python,                          
  basicstyle=\ttfamily\small, 
  numbers=left,                             
  numberstyle=\scriptsize\color{black},  
  stepnumber=1,                   
  numbersep=7pt,                  
  backgroundcolor=\color{white},      
  stringstyle=\color{red},
  showspaces=false,               
  showstringspaces=false,         
  showtabs=false,                 
  tabsize=2,                      
  captionpos=b,                   
  breaklines=true,                
  breakatwhitespace=false,        
  rulecolor=\color{black},        
  framexleftmargin=1mm, framextopmargin=1mm, frame=shadowbox, 
  alsoletter={1234567890},
  otherkeywords={\ , \}, \{},
  keywordstyle=\color{blue},       
  stringstyle=\color{mauve},         
  emph={access,and,break,class,continue,def,del,elif ,else,%
  except,exec,finally,for,from,global,if,import,in,i s,%
  lambda,not,or,pass,print,raise,return,try,while},
  emphstyle=\color{black}\bfseries,
  emph={[2]True, False, None, self},
  emphstyle=[2]\color{blue},
  emph={[3]from, import, as},
  emphstyle=[3]\color{blue},
  upquote=true,
  morecomment=[s]{"""}{"""},
  commentstyle=\color{dkgreen},       
  emph={[4]1, 2, 3, 4, 5, 6, 7, 8, 9, 0},
  emphstyle=[4]\color{blue},
  literate=*{:}{{\textcolor{blue}:}}{1}%
  {=}{{\textcolor{blue}=}}{1}%
  {-}{{\textcolor{blue}-}}{1}%
  {+}{{\textcolor{blue}+}}{1}%
  {*}{{\textcolor{blue}*}}{1}%
  {!}{{\textcolor{blue}!}}{1}%
  {(}{{\textcolor{blue}(}}{1}%
  {)}{{\textcolor{blue})}}{1}%
  {[}{{\textcolor{blue}[}}{1}%
  {]}{{\textcolor{blue}]}}{1}%
  {<}{{\textcolor{blue}<}}{1}%
  {>}{{\textcolor{blue}>}}{1},%
}}{}

\definecolor{orange}{rgb}{1,0.65,0.17}

\providecommand{\abs}[1]{\left\lvert#1\right\rvert}

\def\Q{\mathbb{Q}}

\def\P{\mathbb{P}}
\def\A{\mathbb{A}}

\def\F{\mathbb{F}}

\def\O{\mathcal{O}}

\newcommand{\col}{\,{:}\,}
\newcommand{\hhat}{{\hat h}}

 \DeclareMathOperator{\GL}{GL}

\DeclareMathOperator{\Res}{Res}

\theoremstyle{plain}
\newtheorem{thm}{Theorem}[section]
\newtheorem*{thm*}{Theorem}
\newtheorem{lem}[thm]{Lemma}
\newtheorem{prop}[thm]{Proposition}
\newtheorem{cor}[thm]{Corollary}
\theoremstyle{definition}
\newtheorem{defn}[thm]{Definition}

\newtheorem{exmp}[thm]{Example}
\newtheorem*{exmp*}{Example}

\theoremstyle{remark}
\newtheorem*{rem}{Remark}

\excludeversion{code}

\begin{document}
\maketitle

\begin{abstract}
    We bound the length of the periodic part of the orbit of a preperiodic rational subvariety via good reduction information. This bound depends only on the degree of the map, the degree of the subvariety, the dimension of the projective space, the degree of the number field, and the prime of good reduction. As part of the proof, we extend the corresponding good reduction bound for points proven by the author for non-singular varieties to all projective varieties. Toward proving an absolute bound on the period for a given map, we  the bound between the height and canonical height of a subvariety via Chow forms. This gives the existence of a bound on the number of preperiodic rational subvarieties of bounded degree for a given map. An explicit bound is given for hypersurfaces.
\end{abstract}

\section{Introduction}

    Let $K$ be a number field of degree $\nu = [K\col \Q]$ and $f:\P^N \to \P^N$ be a morphism of degree $d$ defined over $K$. Let $X \subseteq \P^N$ be an irreducible projective subvariety of degree $D$ defined over $K$. We say $X$ is \emph{periodic} if there is an integer $n \geq 1$ such that $f^n(X) = X$, and $X$ is \emph{preperiodic} if $f^m$ is periodic for some integer $m \geq 0$.

    In the case that the dimension of $X$ is $0$, i.e., $X$ is a rational point, Morton-Silverman \cite{Silverman7} conjectured the existence of a constant $C(d,\nu,N)$ bounding the number of rational preperiodic points depending only on the degree of $f$, the degree of $K$, and the dimension.
    While little is known about this conjecture in general, adding an additional hypothesis about primes of good reduction yields the existence of a constant $C(d,\nu,\mathfrak{p},N)$ bounding the number of rational preperiodic points, where $\mathfrak{p}$ is a prime of $K$ where $f$ has good reduction; for $N=1$ see \cite{Benedetto2,Silverman7,Zieve}, for $N>1$ see \cite{Hutz2,Pezda}.

    When the dimension of $X>1$, much less is known. By restricting to coordinate-wise univariate polynomial maps, Medvedev-Scanlon are able to classify the fixed subvarieties \cite{MS}. Studying \'etale maps where $X$ has at least one smooth rational point, Bell-Ghioca-Tucker are able to bound the size of the periodic part of the orbit in terms of $\nu,N,\mathfrak{p}$ using $p$-adic methods \cite{BGT}. It should be noted that this is not a bound as in the Morton-Silverman conjecture because it does not bound the number of rational preperiodic varieties, only provides a bound on the period. However, such a bound on the period for points is a key step in obtaining the overall bound for points.
    Furthermore, it is not at all clear that a bound on the number of rational preperiodic subvarieties is possible. If such a bound exists, it would need to at least depend on the degree of the subvariety since for a homogeneous bivariate polynomial $f$ of degree $d$, a map $F:\P^2 \to \P^2$ of the form $F = (f(x,z),f(y,z),z^d)$ has infinitely many fixed curves of the form $(f^n(x),xz^{d^n-1},z^{d^n})$.

    In this article we prove a bound on the periodic part of the orbit that depends only on $(d,s,N,D,\mathfrak{p})$.
    \begin{thm*}(Theorem \ref{thm_period_bound})
        Let $K$ be a number field. Let $f:\P^N \to \P^N$ be a morphism of degree $d$ defined over $K$. Let $X \subset \P^N$ defined over $K$ be an irreducible periodic subvariety of degree $D$ and codimension $t$ for $f$ with minimal period $n$. Let $\mathfrak{p}\in K$ be a prime of good reduction for $f$. If $\deg(f^{\ell}(X)) = \overline{\deg(f^{\ell}(X))}$ for $0 \leq \ell \leq n$, then there exists a constant $C$ depending only on $d,D,N,t,\mathfrak{p}$ such that
        \begin{equation*}
            n \leq C(d,D,N,t,\mathfrak{p}).
        \end{equation*}
    \end{thm*}
    The basic idea of the method is to move the problem to an endomorphism of a component of the Chow variety and appeal to the similar bound for points from the author's previous work \cite{Hutz2}. The restriction on the degree can be thought of as primes of good reduction for the subvariety and can only not be satisfied for finitely many primes for a given periodic subvariety. However, this is probably not quite enough for applications to problems such as the dynamical Mordell-Land conjecture. To apply these methods there, a more general way is needed to get a bound on the degree of a subvariety in a cycle so that a bound on the period can be obtained entirely independently of properties of $X$.

    The second part of this article examines the canonical height of a subvariety in terms of its Chow form. Starting with work of Nesterenko \cite{Nesterenko} and Philippon \cite{Philippon3}, heights of elimination forms, a version of the more general Chow form, was introduced in transcendence theory. Philippon continued to study the heights of subvarieties via the heights of associated Chow forms and proved many properties of heights and also canonical heights of abelian varieties \cite{Philippon, Philippon4, Philippon2}. Alternatively, using the arithmetic intersection theory of Gillet-Soul\'e \cite{Gillett}, Faltings defined the height of a subvariety $X$ as the intersection of the fundamental class of $X$ with the first Chern class of the canonical hermitian line bundle on $\P^N$ raised to the power $\dim(X)$ \cite{Faltings}.
    Faltings' height is the arithmetic analog of the degree in algebraic geometry.  Bost-Gillet-Soul{\`e} defined the height as the intersection of the fundamental class of $X$ with the $d$-th Chern class of the canonical quotient hermitian bundle on $\P^N$ \cite{BGS}. The arithmetic intersection theory machinery is quite powerful, and they prove many results on the heights of subvarieties. Furthermore, Zhang is able to prove the Bogomolov conjecture for abelian varieties \cite{Zhang3} with this framework. However, less is done with canonical heights. Gubler studies local height and canonical heights in this framework. In this article, we extend Philippon's approach for abelian varieties to general projective morphisms and prove a number of basic results about the canonical height of subvarieties \cite{Gubler}. We define
    \begin{equation*}
        \hhat(X) = \lim_{n \to \infty}\frac{\deg(X)}{\deg(f^n(X))}\frac{h(f^n(X))}{\deg(f^n)},
    \end{equation*}
    where $h(X)$ is the height of the associated Chow form. We summarize our main results in the following theorem.
    \begin{thm}
        Let $f:\P^N \to \P^N$ be a morphism of degree $d \geq 2$ defined over a number field $K$. Let $X$ be an irreducible subvariety of degree $D$ and codimension $t$ defined over $K$.
        \begin{enumerate}
            \item (Theorem \ref{thm_heigt_const}) If $X$ is a hypersurface, then
                \begin{equation*}
                    \abs{h(f(X)) - d\frac{\deg(f(X))}{\deg(X)}h(X)} \leq C(f,N,D)
                \end{equation*}
                for an explicitly computed constant $C$.
            \item (Theorem \ref{thm_can_height})
                The canonical height converges and satisfies the functional equation
                \begin{equation*}
                    \hhat(f(X)) = \frac{\deg(f)\deg(f(X))}{\deg(X)}\hhat(X).
                \end{equation*}
            \item (Theorem \ref{thm_height_diff}) If $X$ is a hypersurface, then
            \begin{equation*}
                \abs{\hhat(X) - h(X)} \leq \frac{CD}{(d-1)d^{N-t-1}},
            \end{equation*}
            where $C$ is from part (1).
        \end{enumerate}
    \end{thm}

    The author would like to thank Tom Tucker for many helpful discussions on this topic and comments on an earlier draft and Lucien Szpiro for discussion related to an explicit bound for his result with Bhatnagar \cite{Bhatnagar} for Lemma \ref{lem_r_bound}.

\section{Preliminaries} \label{sect:prelim}
    Let $f:\P^N \to \P^N$ be a morphism of degree $d$. Since $f$ is also proper, the forward image $f(X)$ is a subvariety. Furthermore, if $X$ is pure codimension $t$ and degree $D$, then $f(X)$ is pure codimension $t$ and degree $d^{N-t}D$.  Note that if $X$ is irreducible, then so is $f(X)$. We can explicitly compute forward images via elimination theory.
    \begin{prop}\label{prop_forward_image}
        Let $f:\P^N \to \P^N$ be a morphism of degree $d$ and $X = V(g_1,\ldots,g_k)$ be a subvariety. Let $I \subset K[x_0,\ldots,x_N,y_0,\ldots,y_N]$ be the ideal
        \begin{equation*}
            I = (g_1(\bar{x}),\ldots,g_k(\bar{x}),y_0-f_0(\bar{x}), \ldots, y_N-f_N(\bar{x})),
        \end{equation*}
        where $\bar{x} = (x_0,\ldots,x_N)$.
        Then the elimination ideal $I_{N+1} = I \cap K[y_0,\ldots,y_N]$ is a homogeneous ideal and
        \begin{equation*}
            f(X) = V(I_{N+1}).
        \end{equation*}
    \end{prop}
    \begin{proof}
        Adapted from \cite[Theorem 8.5.12]{CLO}.

        We first prove that $I_{N+1}$ is homogeneous. Since the polynomials $y_i-f_i(\bar{x})$ are not homogeneous, we introduce weights: each $x_i$ has weight $1$ and $y_i$ has weight $d$. Then $I$ is a weighted homogeneous ideal with degrees $\deg(g_i) = d_i$ and $\deg(y_i-f_i(\bar{x})) = d$. A reduced Groebner basis $G$ with respect to any monomial ordering consists of weighted homogeneous polynomials (\cite[Theorem 8.3.2]{CLO}). For an appropriate lexicographic ordering, $G \cap K[y_0,\ldots,y_N]$ is a basis for $I_{N+1} = I \cap K[y_0,\ldots,y_N]$. Thus, $I$ has a weighted homogeneous basis. Since the basis is in $K[y_0,\ldots,y_N]$, it must also be homogeneous.

        Now we consider the image and work in the product $\P^N \times \P^N$. $I$ is not generated by bi-homogeneous polynomials, so we need to consider the ideal
        \begin{equation*}
            J = (g_1(\bar{x}),\ldots,g_k(\bar{x}), y_if_j(\bar{x}) - y_jf_i(\bar{x})).
        \end{equation*}
        We show that $V(J)$ is the graph of $f(X)$. Let $p \in X$, then $(p,f(p)) \in V(J)$. Conversely, suppose that $(p,q) \in V(J)$. Then we must have $p \in X$ and $q_if(p_j) = q_jf(p_i)$. There is a $j$ with $q_j \neq 0$ and, since $f$ is a morphism, an $i$ with $f(p_i) \neq 0$ so that $q_if_j(p) = q_jf_i(p) \neq 0$ so that $q_i \neq 0$. Let $\lambda = q_i/f_i(p)$ and we have $q = \lambda f(p)$ so that $(p,q)$ is on the graph of $f(X)$. Let $\pi_2:\P^N \times \P^N \to \P^N$ be the projection to the second component. Then \begin{equation*}
            f(X) = \pi_2(V(J)).
        \end{equation*}
        Now we show $\pi_2(V(J)) = V(I_{N+1})$.

        It suffices to work in the affine cone $\A^{N+1} \times \A^{N+1}$ and show that $\pi_2(V(J)) = V(I_{N+1}) \subset \A^{N+1}$. Once we exclude the origin, $q \in \pi_2(V(J))$ if and only if there is some $p \in \A^{N+1}$ such that $q = f(p)$ in $\P^N$ so that there is some $\lambda \neq 0$ such that $q = \lambda f(p)$ in $\A^{N+1}$. If we set $\lambda' = \sqrt[d]{\lambda}$, then $q = f(\lambda'p)$, which is equivalent to $q \in \pi(V(I))$. Since $V(I_{N+1}) \subset \A^{N+1}$ is the smallest variety containing $\pi_2(V(J)) \subset \A^{N+1}$, we have $V(I_{N+1}) = \pi_2(V(J))$.
    \end{proof}
    Since $f$ is flat, the preimage of a degree $D$ pure codimension $t$ subvariety is again a pure codimension $t$ subvariety. We compute the preimage of $X = V(g_1,\ldots,g_k)$ as $f^{-1}(X) =  V(g_1 \circ f,\ldots, g_k \circ f)$ and it has degree $d^tD$. Note that preimages of irreducible varieties are not necessarily irreducible.

\subsection{Chow Forms}
    We will define heights and canonical heights in Section \ref{sect.heights} in terms of the height of the associated Chow form. See \cite[Chapter 3]{GKZ} for more details on Chow forms.

    \begin{defn}
        Let $X$ be an $k$-dimensional subvariety of $\P^N$ of degree $D$. The $(N-k-1)$-dimensional projective subspaces of $\P^N$ meeting $X$ form a hypersurface in the Grassmannian $G(N-k-1,N)$. The homogeneous form defining this hypersurface in Pl\"ucker coordinates is called the \emph{Chow form} of $X$ denoted $Ch(X)$.
    \end{defn}
    The polynomial $Ch(X)$ is degree $D$ and if $X$ is a cycle $\sum n_Y Y$ its Chow form is the hypersurface $\prod Ch(Y)^{n_Y}$. Note that when $X$ is irreducible, its Chow form is the ``forme \'elimiante'' from Philippon \cite{Philippon3}.

    There are explicit algorithms to compute the Chow form of a variety, see for example Dalbec \cite{Dalbec3}.

\section{Bounding the Period} \label{sect.period_bound}
    In this section we show that there is a bound on the period of a subvariety that depends only on the degree of the morphism, the degree of the subvariety, the dimension of the space, the degree of the number field, and a prime of good reduction. The idea is that for a periodic subvariety we can restrict the function to an endomorphism of a subvariety of the Chow variety. Then the problem is about periodic points. Since this restriction may be a singular variety, we need to extend the periodic point theorem from the author's previous work \cite{Hutz2} to singular varieties. An important tool is the ability to extend morphisms of varieties to morphisms of projective space from Bhatnagar-Szpiro and Fakhruddin \cite{Bhatnagar,Fakhruddin}.

    To be able to restrict to a subvariety of the Chow variety, we need to have an upper bound on the degree of an element of the orbit. There are a couple issues that arise here. One is that a given orbit can have elements of different degree within the same orbit.
    \begin{exmp}\label{exmp1}
        For
        \begin{align*}
            f:\P^2(\F_2) &\to \P^2(\F_2)\\
            (x,y,z) &\mapsto (z^2, y^2 + xz + z^2, x^2)
        \end{align*}
        we have the orbit
        \begin{equation*}
            V(y + z) \to V(y^2 + xz) \to V(x + y) \to V(x^2 + y^2 + xz + z^2) \to V(y + z) \to \cdots
        \end{equation*}
    \end{exmp}
    Secondly, there may be finitely many primes where the degree of the reduced subvariety is not the degree of the rational subvariety. If we exclude these primes and have a bound on the period modulo $\mathfrak{p}$, then we have an upper bound on the degree of an element in the orbit.

    Although not explicit, seeing that there is a bound on the reduced period depending only on $\mathfrak{p},d,D,N,t$ is quite simple. Since there are only finitely many residue classes of a map of degree $d$ modulo $\mathfrak{p}$ and only finitely many subvarieties of a given degree $D$, there are only finitely many possible cycles. So there is, in fact, a bound on the period of a subvariety modulo $\mathfrak{p}$ that depends only on $\mathfrak{p},d,D,N,t$.
    \begin{lem} \label{lem_m}
        Let $f:\P^N \to \P^N$ be a morphism of degree $d$ defined over a number field $K$. Let $X \subset \P^N$ be a periodic irreducible subvariety of degree $D$ and codimension $t$ defined over $K$. Let $\mathfrak{p} \in K$ be a prime of good reduction for $f$. Then there exists a constant $m$ depending on $d,N,D,\mathfrak{p},t$ such that the period of $X$ modulo $\mathfrak{p}$ is bounded by $m$.
    \end{lem}
    \begin{proof}
        Given $D,t,N,\mathfrak{p}$, there are only finitely many subvarieties defined over the residue field of degree at most $D$ and codimension $t$ (count for example the coefficients of the Chow form). Also, given $d,N,\mathfrak{p}$, there are only finitely many residue classes of morphisms defined over the residue field of degree $d$. Since $X$ is assumed periodic, its reduction modulo $\mathfrak{p}$ must be one of finitely many possible cycles.
    \end{proof}
    \begin{exmp}
        As an example consider $d=2,K=\Q,p=2,D=1$. To make the computational feasible we assume that the degree of an element in the orbit of a degree $1$ subvariety is $\leq 8$. Then $m=7$.
    \end{exmp}
    If we assume that $\deg(f^n(X)) = \overline{\deg(f^n(X))}$, then Lemma \ref{lem_m} also gives a bound on the degree of the subvarieties in the orbit of $X$ as $d^{m(N-t)}D$. In particular, we can embedded every element of the orbit of $X$ into the same Chow variety and look at the induced morphism. This restriction requires only excluding finitely many additional primes (those which divide the discriminant of the Chow forms of the varieties in the orbit of $X$).

    We will need to restrict to a particular subvariety of the Chow variety, but first we extend the bound on the period for points using good reduction information to morphisms of possibly singular varieties. We do this is two different ways to arrive at two different bounds. First we use a result of Bhatnagar-Szpiro \cite{Bhatnagar}.
    \begin{prop}{\cite[Theorem 1]{Bhatnagar}} \label{prop_BS}
        Let $X$ be a projective variety defined over an infinite field $K$, $L$ a very ample line bundle on $X$ and $f:X \to X$ a morphism such that $f^{\ast}L = L^{\otimes d}$ for some integer $d \geq 2$. Then there exists a positive integer $s$ and a morphism $\phi:\P^m \to \P^m$ extending $f^s$, where $m+1 = \dim_K H^0(X,L)$. Moreover, if the linear system $H^0(X,L)$ is complete, then $f$ extends to $\P^m$.
    \end{prop}
    However, to get a bound independent of $f$ and $X$, we need to show that $s$ can be bounded in terms of more general information.
    \begin{lem}\label{lem_r_bound}
        Let $X$ be a projective variety of dimension $u$ defined over an infinite field $K$, $L$ a very ample line bundle on $X$ and $f:X \to X$ a morphism such that $f^{\ast}L = L^{\otimes d}$ for some integer $d \geq 2$. Embedding $X \subset \P^m$ with $L$ let $D$ be the degree of $X$. There exists a constant $C(u,d,D)$ such that the integer $s$ in Proposition \ref{prop_BS} satisfies
        \begin{equation*}
            s \leq C(u,d,D).
        \end{equation*}
    \end{lem}
    \begin{proof}
        Following the proof of Bhatnagar \cite{Bhatnagar}, we need a uniform bound on $n_0$ from Serre's Vanishing Theorem such that for all $n \geq n_0$ we have $H^1(\P^m, I_X(n))=0$, where $I_X(n)$ is the twist of the sheaf of ideals $I(n)$ of $X$. Following Harsthorne \cite[Proposition 24.4]{Hartshorne2}, $n_0$ is bounded over flat families. To apply this to our situation we consider the Chow variety $Ch(D)$ parameterizing varieties of degree $D$ and dimension $u$ in $\P^m$. We take a flattening decomposition of this Chow variety, see \cite[Lecture 8]{Mumford3} or \cite[Theorem 4.3]{Nitsure}.
        Let $C_0$ be the open set of points of $Ch(D)$ where the universal subvariety $X$ of $\P^m$ is flat. By \cite[Proposition 24.4]{Hartshorne2} there is a uniform $n_0$ for $C_0$. Now consider the family $C_1$, the open set of points of $Ch(D) - C_0$ where $X$ restricted is flat. On this family there is an $n_1$.
        Continue this process for the finitely many possible families (\cite[Theorem 4.3]{Nitsure}) until we have $H^1(\P^m, I_Z(n))=0$ for all $Z \in Ch(D)$. Then the desired $n_0$ is the max these $n_i$.

        Given this uniform $n_0$, $s$ is chosen so that $d^s > \max(D,n_0) \geq \max(\deg(h_i),n_0)$ where $X = V(h_1,\ldots,h_v)$ for some set of homogeneous generating polynomials. The second inequality comes form \cite[Proposition 3]{Heintz} on the degree of generators.
    \end{proof}
    We can now prove the extended theorem.

    \begin{thm}\label{thm_mrpe}
        Let $X$ be an irreducible projective variety of degree $D$ defined over a number field $K$ equipped with a line bundle $L$ giving an embedding $X \hookrightarrow \P^N$. Let $f:X \to X$ be a morphism such that $f^{\ast}L \cong L^{\otimes d}$, $d \geq 2$. Let $\mathfrak{p}$ be a prime of $K$ with residue field $k$ for which $f$ has good reduction.

        Let $P$ be a periodic point of primitive period $n$ for $f$ whose reduction to the residue field $k$ has primitive period $m$.
        Then there is a constant $C$ depending on $d,D,N,\dim(X),\mathfrak{p}$ such that
        \begin{equation*}
            n \leq C(d,D,N,\dim(X),\mathfrak{p})
        \end{equation*}
        where
        \begin{equation*}
            C \leq s \cdot \#X(k)\cdot\#GL_{N+1}(k)\cdot p^e
        \end{equation*}
        where $s$ is from Proposition \ref{prop_BS}, $p = \mathfrak{p} \cap \Q$, and
        \begin{equation*}
            e \leq
            \begin{cases}
              1 + \log_2(v(p)) & p \neq 2\\
              1 + \log_{\alpha}\left(\frac{\sqrt{5}v(2) + \sqrt{5(v(2))^2 + 4}}{2}\right) & p=2,
            \end{cases}
        \end{equation*}
        where $\alpha = \frac{1+\sqrt{5}}{2}$.
    \end{thm}
    \begin{proof}
        By Proposition \ref{prop_BS} and Lemma \ref{lem_r_bound}, there is an integer $s\geq 1$ bounded by a constant depending on $d,D,N,\dim(X)$ such that $f^s$ extends to a morphism on $\P^N$. In other words, there is a map $\phi:\P^N \to \P^N$ making the following diagram commute
        \begin{equation*}
            \xymatrix{X \ar[r]^{f^s}  \ar[d] & X \ar[d]\\
                    \P^N \ar[r]^{\phi} & \P^N.}
        \end{equation*}
        Let $g = \gcd(s,n)$. Applying \cite[Theorem 1]{Hutz2} to $\phi$ we see that
        \begin{equation*}
            \frac{n}{g} = m'rp^e,
        \end{equation*}
        where $m'$ is the minimal period of $f^s(P)$ over $k$, $r$ is the multiplicative order of the multiplier of $f^s(P)$ for $\phi$ over $k$, and $e$ is bounded as in the statement of \cite[Theorem 2]{Hutz2}. In particular,
        \begin{equation*}
            n \leq gm'rp^e \leq smrp^e.
        \end{equation*}
    \end{proof}
    Note that even when $s$ is $1$, this bound is slightly weaker than \cite[Theorem 1]{Hutz2} in the multiplier term, but it has the advantage of applying to singular varieties. Note that \cite[Proposition 2.1]{Fakhruddin} gives an equivalent set of conditions for $H^0(X,L)$ to be complete, i.e., for when $s=1$.

    Since Lemma \ref{lem_r_bound} is not effective, we give a second proof that can provide an effective bound by modifying Fakhruddin's construction in \cite[Proposition 2.1]{Fakhruddin}. Essentially, by results of Mumford \cite[Theorem 1 and 3]{Mumford2}, we are able to provide an embedding where $s=1$ and just need to keep track of the degrees and dimensions. Note that this version also removes the dependency on the dimension of $X$.
    \begin{thm}\label{thm_mrpe_2}
        Let $X$ be a projective variety of degree $D$ defined over a number field $K$ equipped with a line bundle $L$ giving an embedding $X \hookrightarrow \P^N$. Let $f:X \to X$ be a morphism such that $f^{\ast}L \cong L^{\otimes d}$, $d \geq 2$. Let $\mathfrak{p}$ be a prime of $K$ with residue field $k$ for which $f$ has good reduction.

        Let $P$ be a periodic point of primitive period $n$ for $f$ whose reduction to the residue field $k$ has primitive period $m$.
        Then
        \begin{equation*}
            n \leq \#X(k) \cdot \#\GL_{M+1}(k) \cdot p^e \leq \#\P^N(k) \cdot \#\GL_{M+1}(k) \cdot p^e
        \end{equation*}
        where $M = \binom{N+D}{N}-1$, $p = \mathfrak{p} \cap \Q$, and
        \begin{equation*}
            e \leq
            \begin{cases}
              1 + \log_2(v(p)) & p \neq 2\\
              1 + \log_{\alpha}\left(\frac{\sqrt{5}v(2) + \sqrt{5(v(2))^2 + 4}}{2}\right) & p=2,
            \end{cases}
        \end{equation*}
        where $\alpha = \frac{1+\sqrt{5}}{2}$.
    \end{thm}
    \begin{proof}
        Let $V = H^0(X,L)$ and let $i:X \hookrightarrow \P(V)$ be the embedding induced by $L$. For $f$ to extend, we need to satisfy the following conditions from \cite[Proposition 2.1]{Fakhruddin}:
        \begin{enumerate}
            \item The maps $i^{\ast}: H^0(\P(V),\O(n)) \to H^0(X,L^{\otimes n})$ are surjective for all $n \geq 0$.
            \item $i(X)$ is cut out set-theoretically by homogeneous forms of degree $\leq d$.
        \end{enumerate}
        From Mumford \cite[Theorem 1 and 3]{Mumford2}, embedding $X$ in a larger projective space with the $\deg(X)$-uple Veronese embedding and replacing $L$ by a power at most $\dim(X) + 1$ results in satisfying both properties.
        We are embedding $X$ in projective space of dimension $M = \binom{N+D}{N}-1$. We apply \cite[Theorem 1]{Hutz2} to the resulting map $\psi:\P^M \to \P^M$, which extends $f$ to have any rational periodic points with minimal period $n$ satisfying
        \begin{equation*}
            n \leq \#X(k)\cdot \#\GL_{M+1}(k) \cdot p^e.
        \end{equation*}
    \end{proof}

    We can now prove the existence of a bound for the minimal period of subvarieties based on good reduction information.

    \begin{thm} \label{thm_period_bound}
        Let $K$ be a number field. Let $f:\P^N \to \P^N$ be a morphism of degree $d$ defined over $K$. Let $X \subset \P^N$ defined over $K$ be an irreducible periodic subvariety of degree $D$ and codimension $t$ for $f$ with minimal period $n$. Let $\mathfrak{p}$ be a prime of good reduction for $f$. If $\deg(f^{\ell}(X)) = \overline{\deg(f^{\ell}(X))}$ for $0 \leq \ell \leq n$, then there is a constant $C$ depending on $d,D,N,t,\mathfrak{p}$ such that
        \begin{equation*}
            n \leq C(d,D,N,t,\mathfrak{p}).
        \end{equation*}
    \end{thm}
    \begin{proof}
        Let $m$ be the period of $\overline{X}$ over the residue field $k$ (which is bounded in Lemma \ref{lem_m}). The morphism $f^m$ induces a morphism $\phi$ from the Chow variety of degree $D$ (i.e., $\P^M$ for $M=\binom{N+d}{N}$) to the Chow variety of degree $d^{m(N-t)}D$. We have $\deg(\phi) = d^{mN}$.
        However, there is a closed subvariety of the Chow variety on which $f^m$ preserves the degree (which contains $X$). This is the subvariety where the image Chow form is a power. In particular, if we take the image of a symbolic Chow form $T$ of degree $D$, then the discriminant of the resulting Chow form of degree $d^{m(N-t)}D$ gives the closed subvariety where the image degree is strictly less than $d^{m(N-t)}D$. We can compute this discriminant as \cite[Prop 1.7, p434]{GKZ}:
         \begin{equation*}
            Z_{1}(\phi) = \Delta = \Res\left(\frac{\partial^{d-1} Ch(\phi(T))}{\partial x_0^{d-1}},\ldots,\frac{\partial^{d-1} Ch(\phi(T))}{\partial x_M^{d-1}}\right).
         \end{equation*}
         We want the subvariety where the image degree is $D$. In particular, we need the image to be a $d^{m(N-t)}$-th power. So we take the subvariety defined by
         \begin{equation*}
            Z_{m(N-t)}(\phi) = \Res\left(\frac{\partial^{d^{m(N-t)}-1} Ch(\phi(T))}{\partial x_0^{d^{m(N-t)}-1}},\ldots,\frac{\partial^{d^{m(N-t)}-1} Ch(\phi(T))}{\partial x_M^{d^{m(N-t)}-1}}\right).
         \end{equation*}
         On each irreducible component $Y_1,\ldots,Y_v$ of this subvariety, the map $\phi$ induces a map $\psi_i:Y_i \to \P^M$ of degree $\frac{d^{mN}}{d^{m(N-t)}} = d^{mt}$. Note that the $Y_i$ may be singular. Assume, after possibly renaming, that $Ch(X) \in Y_1$. This map is defined by a single set of homogeneous polynomials, but it is not yet a self-map of varieties.
         Consider the following construction. Since $Y_1$ is irreducible,   the intersection $\psi_1(Y_1)\cap Y_1$ is the intersection of two irreducible subvarieties, so either is $Y_1$ or has dimension strictly less than $Y_1$. If it is all of $Y_1$, then we have our self-map. If the dimension is strictly less, consider the component $Y_{1,1}$ of $\psi_1(Y_1)$ which contains $X$. We again consider $\psi_1(Y_{1,1}) \cap Y_1$. Again, we either get a self-map or a strictly smaller dimension that $Y_{1,1}$. This process will terminate after finitely many steps in an irreducible variety $\tilde{Y_1}$.
         Since $Ch(X) \in \tilde{Y_1} \subset \P^M$, this intersection is nonempty and we have
         \begin{equation*}
            \psi_1: \tilde{Y_1} \to \tilde{Y_1}
         \end{equation*}
         is a morphism of varieties.
         To get our final constant depending only on $d,D,N,t,\mathfrak{p}$ we need to know that the degree of $\tilde{Y_1}$ can be bounded only in terms of $d,D,N,t,\mathfrak{p}$. We can bound the degree of $\tilde{Y_1}$ with a B\'ezout type bound. The number of steps in the process is at most $\dim(Y_1) = M-1$. The degree of $Y_1 = (M+1)((d^MD-d^m)(N-t)+1)$. 
         If it takes $j$ steps to get a self map, the degree of each irreducible component containing $X$ is bounded above by
         \begin{equation*}
            \deg(Y_{1,i}) \leq (d^{mt})^{i}(M+1)((d^MD-d^m)(N-t)+1) \quad \text{for } 0 \leq i \leq j-1,
         \end{equation*}
         where we let $Y_{1,0} = Y_1$.
         The degree at each stage depends on the degree of $Y_1$ and $\psi_1$.
         Since for each step we are intersecting an irreducible variety with a hypersurface and the resulting dimension in one less, we can apply B\'ezout's Theorem on the degree of intersections. Since we are taking only one irreducible component, this provides an upper bound
         \begin{equation*}
            \deg(\tilde{Y_1}) \leq \prod_{i=0}^{j-1} (d^{mt})^{i}(M+1)((d^MD-d^m)(N-t)+1).
         \end{equation*}
         Most importantly, the number of steps is bounded by $M-1$, and this degree bound depends only on $d,D,N,t,\mathfrak{p}$. We could remove the dependency on $\mathfrak{p}$ for this bound by removing the $d^m$ terms and making the bound larger.

         We apply Theorem \ref{thm_mrpe} to $\psi_1$ on $\tilde{Y_1}$ to get a constant such that
         \begin{equation*}
            n \leq C(d,D,N,t,\mathfrak{p}).
         \end{equation*}
    \end{proof}
    \begin{exmp}
        Consider the map
        \begin{equation*}
            f(x,y,z) = (x^2,y^2+z^2,z^2).
        \end{equation*}
        It is clear that $V(z) = \{(x,,y,z) \in \P^2 : x=0\}$ is fixed, so we'll consider linear subvarieties
        \begin{equation*}
            T = V(a_0x + a_1y + a_2z).
        \end{equation*}
        We have
        \begin{align*}
            f(T) &= V(a_0^4x^2 - 2a_0^2a_1^2xy + a_1^4y^2 + (2a_0^2a_1^2 - 2a_0^2a_2^2)xz + (-2a_1^4 - 2a_1^2a_2^2)yz + (a_1^4 + 2a_1^2a_2^2 + a_2^4)z^2).
        \end{align*}
        So we have the map
        \begin{align*}
            \phi:\P^2 &\to \P^5\\
            (a_0,a_1,a_2) & \mapsto (a_0^4, -2a_0^2a_1^2, a_1^4, 2a_0^2a_1^2 - 2a_0^2a_2^2, -2a_1^4 - 2a_1^2a_2^2, a_1^4 + 2a_1^2a_2^2 + a_2^4]).
        \end{align*}
        We want to know the subvarieties that remain linear under $f$, so we compute
        \begin{align*}
            Z_1 &= \Res\left(\frac{\partial \phi(T)}{\partial x},\frac{\partial \phi(T)}{\partial y},\frac{\partial \phi(T)}{\partial z}\right)\\
            &=-32a_0^4a_1^4a_2^4
        \end{align*}
        and we see that $Z_1$ has three components:
        \begin{equation*}
            Y_1 = V(a_0), \qquad Y_2 = V(a_1), \qquad Y_3 = V(a_2).
        \end{equation*}
        On $Y_1 \supset X$ we have
        \begin{align*}
            \psi:Y_1 &\to \P^2\\
            (0,a_1,a_2) &\mapsto (0, a_1^2,-a_1^2 - a_2^2),
        \end{align*}
        which we can extend to
        \begin{align*}
            \psi_1:Y_1 &\to \P^2\\
            (a_0^2,a_1,a_2) &\mapsto (a_0^2, a_1^2,-a_1^2 - a_2^2).
        \end{align*}
        Notice that $\psi_1$ is an endomorphism of $Y_1$.
\begin{code}
\begin{python}
set_verbose(None)
R.<a_0,a_1,a_2>=PolynomialRing(QQ)
P.<x,y,z>=ProjectiveSpace(FractionField(R),2)
H=End(P)
f=H([x^2,y^2+z^2,z^2])
X=P.subscheme([a_0*x+a_1*y+a_2*z])
T=f(X).defining_polynomials()[0]*a_0^4
T.coefficients()

P.coordinate_ring().macaulay_resultant([T.derivative(x),T.derivative(y),T.derivative(z)])

set_verbose(None)
R.<a_0,a_1,a_2>=PolynomialRing(QQ)
P.<x,y,z>=ProjectiveSpace(FractionField(R),2)
H=End(P)
f=H([x^2,y^2+z^2,z^2])
X=P.subscheme([a_1*y+a_2*z])
T=f(X).defining_polynomials()[0]*a_1^2
T
\end{python}
\end{code}
     \end{exmp}

\section{Explicit Heights and Canonical Heights} \label{sect.heights}

    In this section we prove basic properties of heights and canonical heights of subvarieties as defined by the height of the associated Chow form. Recall that the height of a polynomial is the maximum of the heights of its coefficients.

    \begin{defn}
       Given a subvariety $X \subset \P^N$, we define the height of $X$ as
       \begin{equation*}
            h(X) = h(Ch(X)),
       \end{equation*}
       where $Ch(X)$ is associated Chow form defined in Section \ref{sect:prelim}.
    \end{defn}
    The height satisfies the following properties.
    \begin{prop}
        Let $X \subset \P^N$ be a subvariety of degree $D$.
        \begin{enumerate}
            \item $h(X) \geq 0$.
            \item There are only finitely many subvarieties of bounded height and bounded degree over a number field of bounded degree.
            \item Philippon's height \cite{Philippon}, Faltings height \cite{Faltings2}, Bost-Gillet-Soul\'e's height \cite{BGS}, and $h$ are all equivalent.
        \end{enumerate}
    \end{prop}
    \begin{proof}
        The first two properties are obvious from the defnition of $h(X)$ as the height of $Ch(X)$ as a polynomial. The third property can be found in Bost-Gillet-Soul\'e \cite{BGS}; see for example Proposition 4.1.2, remark after Theorem 4.2.3, and remarks starting with 4.3.12.
    \end{proof}

\subsection{Height bounds of forward images}
    It is well known that for a morphism $f$ and a point $P$ that
    \begin{equation*}
        \abs{h(f(P))- dh(P)} \leq C
    \end{equation*}
    for an explicitly computable constant $C$ \cite[Theorem 3.11]{Silverman10}. Similarly \cite[Prop 3.2.2]{BGS}, for a subvariety $X$ of codimension $t$, there exists a (non-explicit) constant $C$ such that
    \begin{equation*}
        \abs{h(f(X)) - \deg(f)\frac{\deg(f(X))}{\deg(X)}h(X)} \leq C\deg(X).
    \end{equation*}
    We first prove a lemma on the size of coefficients occurring in resultants.
    \begin{lem}\cite[Proposition 7]{Wustholz} \label{lem_Wustholz}
    For $N+1$ multi-homogeneous polynomials in $N+1$ variables of degrees $D,d,\ldots,d$,
        \begin{equation*}
            H(Res_{D,d,\ldots,d}) \leq exp\left(\binom{D + Nd^N +1}{N}\right)\binom{D + Nd^N +1}{N}!.
        \end{equation*}
    \end{lem}

    \begin{lem} \label{lem_morphism}
        The map induced on the Chow coordinates in codimension $t$ is a morphism of degree $d^{N-t+1}$.
    \end{lem}
    \begin{proof}
        A flat proper map induces a homomorphism between the Chow groups. In particular, a codimension $t$ subvariety maps to a codimension $t$ subvariety. We can compute the image subvariety via Proposition \ref{prop_forward_image}. We think of this map as acting on a projective point representing the coefficients of the Chow form (i.e., the Chow coordinates). The image coordinates will be polynomials in the original coordinates. Since we know each image is a codimension $t$ subvariety, this map is a morphism.

        To compute its degree, we simply count the number of inverse images under the pullback map. If $\deg(X) = D$, then $\deg(f^{-1}(X))$ is $dD$. Generically, a degree $\frac{D}{d^{N-t}}$ subvariety maps to a degree $D$ subvariety, so the number of inverse images is
        \begin{equation*}
            dD\frac{d^{N-t}}{D} = d^{N-t+1}.
        \end{equation*}
    \end{proof}

    \begin{thm}\label{thm_heigt_const}
        Let $f:\P^N \to \P^N$ be a morphism of degree $d \geq 2$. Let $X \subset \P^N$ be a hypersurface of degree $D$. Then there exists an explicitly computable constant such that
        \begin{equation*}
            \abs{h(f(X)) - d\frac{\deg(f(X))}{\deg(X)}h(X)} \leq C(f,N,D).
        \end{equation*}
        Moreover,
        \begin{align*}
            C &= \frac{\deg(f(X))/\deg(X)}{d^{N-1}}\left(\log(\tau(\deg(f(X)))) + \log(\binom{\tau(D)-1+e(D)-d^N}{e(D)-d^N}) \right.\\
            & + [4\tau(D)(\tau(D)+1)(d^N)^{\tau(D)}(2d^{N-1}Dh(f) + \log(N) + \log(C) + \log(\tau(\deg(f(X))))\\
            &\left.+ (\tau(D)+7)\log(\tau(D)+1)d^N)]\right)
        \end{align*}
    \end{thm}
    \begin{proof}
        Since $X$ is a degree $D$ hypersurface, it is the vanishing locus of a homogeneous degree $D$ polynomial, $g$.
        Let $f(\overline{x}) = (f_0(\overline{x}),\ldots, f_N(\overline{x}))$ be a tuple of homogenous polynomials in the variables $\overline{x} = (x_0,\ldots,x_N)$. For a second set of variables $\overline{y} = (y_0,\ldots,y_N)$, define the tuple
        \begin{equation*}
            J = (y_jf_i(\overline{x}) - y_if_j(\overline{x}), g(\overline{x})) \text{ for }0 \leq i < j \leq N.
        \end{equation*}
        The equations $y_jf_i(\overline{x})-y_if_j(\overline{x})$ are degree $d$ in $\overline{x}$ and $g(\overline{x})$ is degree $D$. The ideal generated by the generalized resultant $\Res_{D,d,\ldots,d}(J)$ in terms of $x$ is the forward image of $X$ by $f$. Recall that the generalized resultant $\Res_{d_i}$ is homogeneous of degree $\frac{\prod d_i}{d_j}$ in the coefficients of the equation corresponding to $d_j$ for each $j$. In this particular case, the resultant has degree $d^N$ in the coefficients of $g$ and degree $d^{n-1}D$ in the coefficients of each of the $y_jf_i(\overline{x})-f_j(\overline{x})y_i$.

        It is possible that $\frac{\deg(f(X))}{\deg(X)} < d^{N-1}$, i.e., that the resultant is a power. In particular, the $\frac{d^{N-1}}{\deg(X)/\deg(f(X))}$-root of the bound on the resultant is a bound on $h(f(X))$.
        We see that an upper bound is given by
        \begin{equation*}
            h(\Res) \leq [d^Nh(X) + d^{N-1}D(2h(f)) + \log(N)] + \log(C),
        \end{equation*}
        where $C$ is the max coefficient in the resultant. From Lemma \ref{lem_Wustholz}, we have
        \begin{equation*}
            \log(C) \leq \binom{D + Nd^N +1}{N}+ \log\left(\binom{D + Nd^N +1}{N}!\right).
        \end{equation*}
        Therefore,
        \begin{align*}
             &h(f(X)) = \frac{\deg(f(X))/\deg(X)}{d^{N-1}}h(\Res)\\
             &\leq \frac{\deg(f(X))/\deg(X)}{d^{N-1}}\left([d^Nh(X) + 2d^{N-1}Dh(f)+ \log(N)] + \binom{D + Nd^N +1}{N}+\log\left(\binom{D + Nd^N +1}{N}!\right)\right).
        \end{align*}
        In particular,
        \begin{align*}
            h(f(X)) &- d\frac{\deg(f(X))}{\deg(X)}h(X)\\& \leq \frac{\deg(f(X))/\deg(X)}{d^{N-1}}\left(2d^{N-1}Dh(f) + \log(N)+ \binom{D + Nd^N +1}{N}+\log\left(\binom{D + Nd^N +1}{N}!\right)\right)
        \end{align*}

        The lower bound is somewhat more complicated. The morphism
        $f$ induces a map on Chow varieties. These Chow varieties can be thought of as points with their coefficients as coordinates in some projective space. For ease of notation, denote $\deg(f(X))=D'$. This induced map is a morphism of degree $d^N$ (Lemma \ref{lem_morphism})
        \begin{equation*}
            \phi: \P^{\tau(D)-1} \to \P^{\tau(D')-1},
        \end{equation*}
        where $\tau(v) = \binom{N+v}{v}$. For a point $P$ we know that
        \begin{equation*}
            \abs{h(\phi(P)) - d^Nh(P)} \leq C_2
        \end{equation*}
        for an explicitly computable (Nullstellensatz) constant $C_2$ \cite[Theorem 3.11]{Silverman10}. As with the upper bound, if the degree of $f(X)$ is not the full $d^{N-1}\deg(X)$, we get the actual bound from taking a root. We next bound the constant.

        The size of the domain $\tau(D)-1$ is one less than the number of monomials in a degree $D$ hypersurface which is
        \begin{equation*}
            \tau(D) = \binom{N+D}{N}
        \end{equation*}
        Similarly for $\tau(D')$. Since the degree of the map is $d^{N}$, from \cite[\S 10]{Wustholz} define
        \begin{equation*}
            e(D) = \tau(D)d^N - (\tau(D)-1) + 1
            = (\tau(D) -1)d^N + 2.
        \end{equation*}

        From \cite[Theorem 3.11]{Silverman10} and \cite[Theorem 1]{KPA}, we have
        \begin{align*}
            H(P)^{d^N} &\leq \tau(D') \binom{\tau(D)-1+e(D)-d^N}{e(D)-d^N} \max(H(g_i)) H(\phi(P))\\
            h(g) &\leq [
            4\tau(D)(\tau(D)+1)(d^N)^{\tau(D)}(h(\phi) + \log(\tau(D')) + (\tau(D)+7)\log(\tau(D)+1)d^N)]\\
           d^Nh(P) &\leq \log(\tau(D')) + \log(\binom{\tau(D)-1+e(D)-d^N}{e(D)-d^N}) + [
            4\tau(D)(\tau(D)+1)(d^N)^{\tau(D)}(h(\phi)\\ &+ \log(\tau(D')) + (\tau(D)+7)\log(\tau(D)+1)d^N)] + h(\phi(P))
        \end{align*}
        So we need an upper bound on $h(\phi)$. We can get this from an upper bound on the resultant from above (without the $h(V)$ part) as
        \begin{equation*}
            h(\phi) \leq 2d^{N-1}Dh(f) + \log(N) + \log(C),
        \end{equation*}
        where $C$ is the Wustholz upper bound on the max coefficient in the resultant (Lemma \ref{lem_Wustholz}).

        So we have
        \begin{align*}
             d^Nh(P) & \leq \log(\tau(D')) + \log(\binom{\tau(D)-1+e(D)-d^N}{e(D)-d^N}) + [
            4\tau(D)(\tau(D)+1)(d^N)^{\tau(D)}\\ & (2d^{N-1}Dh(f) + \log(N) + \log(C) + \log(\tau(D')) + (\tau(D)+7)\log(\tau(D)+1)d^N)] + h(\phi(P)).
        \end{align*}
        Now we take the $\frac{d^{N-1}}{\deg(f(X))/\deg(X)}$ root of both sides (except for the $h(\phi(P))$ term on the right) to get
        \begin{align*}
             d\frac{\deg(f(X))}{\deg(X)}h(P) & \leq \frac{\deg(f(X))/\deg(X)}{d^{N-1}}\left(\log(\tau(D')) + \log(\binom{\tau(D)-1+e(D)-d^N}{e(D)-d^N})\right.\\
             &+ [4\tau(D)(\tau(D)+1)(d^N)^{\tau(D)}
            (2d^{N-1}Dh(f) + \log(N) + \log(C) + \log(\tau(D'))\\
             &\left.+ (\tau(D)+7)\log(\tau(D)+1)d^N)]\right) + h(\phi(P)).
        \end{align*}

        This difference is clearly larger than the upper bound, so this is the desired constant.
    \end{proof}

    We now want to simplify the form of the constant to get a rough estimate of growth. We use two main tools:
    \begin{equation*}
        \binom{n+k}{k} \leq n^k
    \end{equation*}
    and
    \begin{equation*}
        n! \leq n^n \quad \text{so that} \quad \log(n!) \leq n \log(n).
    \end{equation*}

    \begin{cor}
        Let $f:\P^N \to \P^N$ be a morphism for degree $d \geq 2$. Let $X$ be a hypersurface of degree $D$. We have the following bound
        \begin{equation*}
            \abs{h(f(X)) - \deg(f)\frac{\deg(f(X))}{\deg(X)}h(X)} \leq O(D^{3N}\log(D)(d^N)^{D^N})
        \end{equation*}
        where the constant depends on $f,d,N$.
    \end{cor}
    \begin{proof}
        We see that an upper bound is given by
        \begin{equation*}
            h(f(X)) - \deg(f)\frac{\deg(f(X))}{\deg(X)}h(X) \leq [d^{N-1}D(2h(f)) + \log(N)] + \log(C)
        \end{equation*}
        where $C$ is bounded by
        \begin{align*}
            \log(C) &\leq \binom{D + Nd^N +1}{N}+ \log\left(\binom{D + Nd^N +1}{N}!\right)\\
            &\leq (D+Nd^N +1)^N + D^NN\log(D) = O(D^N(1 + N \log(D))).
        \end{align*}
        So we have
        \begin{equation*}
            h(f(X)) - \deg(f)\frac{\deg(f(X))}{\deg(X)}h(X) \leq C_2(d,N,f)D + C_3(d,N)D^NN\log(D).
        \end{equation*}

        For the lower bound,
        \begin{equation*}
            \tau(D) = \binom{N+D}{N} \leq (D+N)^N = O(D^N).
        \end{equation*}
        The size of the codomain
        \begin{equation*}
            \tau(\deg(f(X))) \leq \binom{d^{N-1}D + N}{N} \leq (d^{N-1}D + N)^N = O(D^N)
        \end{equation*}
        and constant
        \begin{equation*}
            e(D) = \left(\binom{N+D}{N} + 1\right)(d^N - 1) + 1 = O(D^N)
        \end{equation*}
        give
        \begin{align*}
             \deg(f)\frac{\deg(f(X))}{\deg(X)}h(P) & \leq \frac{\deg(f(X))/\deg(X)}{d^{N-1}}\left(\log(\tau(D')) + \log(\binom{\tau(D)-1+e(D)-d^N}{e(D)-d^N})\right.\\
            &+ [4\tau(D)(\tau(D)+1)(d^N)^{\tau(D)}(2d^{N-1}Dh(f) + \log(N) + \log(C) + \log(\tau(D'))\\
            &\left.+ (\tau(D)+7)\log(\tau(D)+1)d^N)]\right) + h(\phi(P)).
        \end{align*}
        Estimating in the generic case ($\deg(f(X))=d^{N-1}D$),
        \begin{equation*}
             d^Nh(P) \leq O(\log(D^N) + \log((D^N)^{D^N}) +             O(D^{3N}\log(D)(d^N)^{D^N}) + h(\phi(P)).
        \end{equation*}
    \end{proof}
    We can use the above explicit bound to get an explicit bound on the height of a preperiodic subvariety of degree $D$ by taking an upper bound of the $D$ that can occur in the cycle.

    \begin{exmp}
        We compute $C$ for $N=2, d=2, D=1$.

        We first compute the upper bound assuming the subvariety remains linear, even though we know it is smaller than the lower bound.
        \begin{align*}
            \abs{h(f(X)) - dh(X)} &\leq \frac{1}{2}\left(2d^{N-1}Dh(f) + \log(N)+ \binom{D + Nd^N +1}{N}+\log\left(\binom{D + Nd^N +1}{N}!\right)\right)\\
            &= \frac{1}{2}\left(4h(f) + \log(2) + \binom{1+8+1}{2} + \log(\binom{10}{2}!)\right)\\
            &= 2h(f) + \frac{175}{2}.
        \end{align*}

        We compute the lower bound.
        \begin{align*}
            \tau(D) &= \binom{N+D}{N} = \binom{4}{2} = 6\\
            e(D) &= \left(\binom{4}{2} + 1\right)(4-1)+1 = 22\\
            \log(C_1) &\leq \binom{10}{2}+ \log\left(\binom{10}{2}!\right) \leq 175.
        \end{align*}
        \begin{align*}
            \abs{h(f(X)) - dh(X)} & \leq \frac{1}{2}\left(\log(\tau(dD)) + \log(\binom{\tau(D)-1+e(D)-d}{e(D)-d}) + [
            4\tau(D)(\tau(D)+1)(d)^{\tau(D)}\right.\\ & \left.(2d^{N-1}Dh(f) + \log(N) + \log(C) + \log(\tau(dD)) + (\tau(D)+7)\log(\tau(D)+1)d^N)]\right)\\
            &= \frac{1}{2}\left(\log(2) + \log(\binom{18-2}{12-2}) + [
            24(7)(2)^{6}(4h(f) + \log(2) + 175 + \log(2) + (13)\log(7)4)]\right)\\
            &= \frac{1}{2}\left(\log(2) + \log(8008) + 10752(4h(f) + 2\log(2) + 175 + 52\log(7))\right)\\
            &\leq 21504h(f) + 1492241.
        \end{align*}
    \end{exmp}
    While we have achieved an explicit bound, it is not so useful in practice.

\subsection{Canonical Heights}
    Gubler proves the existence of a canonical height and local canonical height in the language of arithmetic intersection theory \cite{Gubler,Gubler2}. We take the more direct approach with Chow forms in order to obtain a height difference bound and, thus, an upper bound on the height of a preperiodic subvariety.

    We define the canonical height as follows.
    \begin{defn}
        Let $X \subset \P^N$ be a subvariety of codimension $t$. Let $f:\P^N \to \P^N$ be a morphism. Define
        \begin{equation*}
            \hhat(X) = \lim_{n \to \infty}\frac{\deg(X)}{\deg(f^n(X))}\frac{h(f^n(X))}{\deg(f^n)}.
        \end{equation*}
    \end{defn}

    \begin{thm}\label{thm_can_height}
        The canonical height converges and satisfies the functional equation
        \begin{equation*}
            \hhat(f(X)) = \frac{\deg(f)\deg(f(X))}{\deg(X)}\hhat(X).
        \end{equation*}
    \end{thm}
    \begin{proof}
        For convergence, we will show the sequence is Cauchy. Assume that $n>m\geq 0$. We will use the fact that
        \begin{equation*}
            \deg(f^{n+1}(X)) \leq d^{N-t}\deg(f^n(X))
        \end{equation*}
        and the existence of a constant $C$ such that
        \begin{equation*}
            \abs{h(f(X)) - \deg(f)\frac{\deg(f(X))}{\deg(X)}h(X)} \leq C\deg(X).
        \end{equation*}
        \begin{align*}
            \abs{\frac{\deg(X)}{\deg(f^n(X))}\frac{h(f^n(X))}{d^n} - \frac{\deg(X)}{\deg(f^m(X))}\frac{h(f^m(X))}{d^m}}
            &= \sum_{i=m+1}^n \frac{\deg(X)h(f^i(X))}{\deg(f^i(X))d^i} - \frac{\deg(X)h(f^{i-1}(X))}{\deg(f^{i-1}(X))d^{i-1}}\\
            &=\sum_{i=m+1}^n \frac{\deg(X)}{d^i}\left(\frac{h(f^i(X))}{\deg(f^i(X))} - d\frac{h(f^{i-1}(X))}{\deg(f^{i-1}(X))}\right)\\
            &\leq \sum_{i=m+1}^n \frac{\deg(X)}{\deg(f^{i-1}(X))d^i}\left(\frac{h(f^i(X))}{d^{N-t}}- dh(f^{i-1}(X))\right)\\
            &= \sum_{i=m+1}^n \frac{\deg(X)}{\deg(f^{i-1}(X))d^{N-t+i}}\left(h(f^i(X))- d^{N-t+1}h(f^{i-1}(X))\right)\\
            &\leq \sum_{i=m+1}^n \frac{C\deg(X)^2}{\deg(f^{i-1}(X))d^id^{N-t}}\\
            &\leq \sum_{i=m+1}^n \frac{C\deg(X)^2}{d^i}\\
            &\leq C\deg(X)^2\sum_{i=m+1}^{\infty} \frac{1}{d^i}\\
            &=\frac{C\deg(X)^2}{d^m(1-1/d),}
        \end{align*}
        which goes to $0$ as $m \to \infty$.

        For the functional equation, we compute
        \begin{align*}
            \hhat(f(X)) &= \lim_{n \to \infty} \frac{\deg(f(X))}{\deg(f^n(f(X)))}\frac{h(f^n(f(X)))}{\deg(f^n)}\\
            &=\lim_{n \to \infty} \frac{\deg(f(X))}{\deg(f^n(f(X)))}\frac{h(f^n(f(X)))}{\deg(f^n)}\frac{\deg(f)\deg(X)}{\deg(f)\deg(X)}\\
            &=\lim_{n \to \infty} \deg(f)\frac{\deg(f(X))}{\deg(X)}\frac{\deg(X)}{\deg(f^{n+1}(X))}\frac{h(f^{n+1}(X))}{\deg(f^{n+1})}\\
            &=\deg(f)\frac{\deg(f(X))}{\deg(X)}\lim_{n \to \infty} \frac{\deg(X)}{\deg(f^{n+1}(X))}\frac{h(f^{n+1}(X))}{\deg(f^{n+1})}\\
            &=\deg(f)\frac{\deg(f(X))}{\deg(X)}\hhat(X).
        \end{align*}
    \end{proof}
    \begin{rem}
        If we normalize the canonical height by the degree of $X$, we have the more visually appealing
        \begin{equation*}
            \frac{\hhat(f(X))}{\deg(f(X))} = d\frac{\hhat(X)}{\deg(X)}.
        \end{equation*}
    \end{rem}

    An immediate corollary of the functional equation is the following.
    \begin{cor}
        Preperiodic subvarieties have canonical height $0$.
    \end{cor}

    With the $C$ from Theorem \ref{thm_heigt_const} we can have a bound between the height and canonical height of a subvariety.
    \begin{thm} \label{thm_height_diff}
        With the previous constant $C$ we have
        \begin{equation*}
            \abs{\hhat(X) - h(X)} \leq \frac{CD}{(d-1)d^{N-t-1}}.
        \end{equation*}
    \end{thm}
    \begin{proof}
        We have
        \begin{equation*}
            h(f^n(X)) \leq d^n\frac{D_n}{D_0}h(X) + d^{n-1}C\frac{D_nD_0}{D_1} + d^{n-2}C\frac{D_nD_1}{D_{2}} + \cdots + dC\frac{D_nD_{n-2}}{D_{n-1}} + CD_n
        \end{equation*}
        so that
        \begin{align*}
            \frac{Dh(f^n(X))}{d^nD_n} &\leq h(X) + \frac{C}{d}\frac{D_0^2}{D_1} + \frac{C}{d^2}\frac{D_0D_1}{D_2} + \cdots + \frac{C}{d^{n-1}}\frac{D_0D_{n-2}}{D_{n-1}} + \frac{C}{d^n}D_0\\
            &\leq h(X) + CD_0\left(\frac{D_0}{dD_1} + \frac{D_1}{d^2D_2} + \cdots + \frac{D_{n-2}}{d^{n-1}D_{n-1}} + \frac{1}{d^n}\right)\\
            &\leq h(X) + CD_0\left(\frac{1}{dd^{N-t}} + \frac{1}{d^2d^{N-t}} + \cdots + \frac{1}{d^{n-1}d^{N-t}} + \frac{1}{d^n}\right)\\
            &\leq h(X) + CD_0\frac{1}{d^{N-t}}\left(\frac{1}{d} + \frac{1}{d^2} + \cdots + \frac{1}{d^{n-1}} + \frac{1}{d^n}\right)\\
            &\leq h(X) + CD_0\frac{1}{d^{N-t}}\frac{1}{1-1/d}\\
            &\leq h(X) + \frac{CD}{(d-1)d^{N-t-1}}.
        \end{align*}
    \end{proof}

    \begin{cor}
        Given $f:\P^N \to \P^N$, a morphism of degree $d$ defined over a number field $K$, there are only finitely many preperiodic rational subvarieties of degree at most $D$ defined over $K$.
    \end{cor}

    \begin{proof}
        A preperiodic subvariety has canonical height $0$, so there is a height bound on preperiodic subvarieties of degree at most $D$. There are only finitely many rational subvarieties of bounded degree and height.
    \end{proof}

\bibliography{masterlist}
\bibliographystyle{plain}

\end{document}